\documentclass[11pt,oneside]{amsart}

\usepackage[T1]{fontenc} 
\usepackage[utf8]{inputenc}  
\usepackage[english]{babel}
\usepackage{vaucanson-g}
\allowdisplaybreaks
\usepackage{pgf,tikz}
\usetikzlibrary{arrows,automata, shapes, positioning}
\usetikzlibrary{decorations}
\usetikzlibrary{snakes}
\usetikzlibrary{arrows.meta}
\usetikzlibrary{decorations.pathmorphing}
\usepackage{float}
\usepackage{cases}

\usepackage{vmargin}
\setmarginsrb{3cm}{2cm}{3cm}{2cm}{0cm}{2cm}{0cm}{1cm}

\theoremstyle{plain}
\newtheorem{theorem}{Theorem}
\newtheorem{lemma}[theorem]{Lemma}
\newtheorem{corollary}[theorem]{Corollary}
\newtheorem{proposition}[theorem]{Proposition}
\theoremstyle{definition}

\newtheorem{example}[theorem]{Example}
\newtheorem{remark}[theorem]{Remark} 
\numberwithin{equation}{section}

\newcommand{\R}{\mathbb R}
\newcommand{\N}{\mathbb N}
\newcommand{\rep}{\mathrm{rep}}
\newcommand{\val}{\mathrm{val}}
\newcommand{\lex}{\mathrm{lex}}
\newcommand{\Fac}{\mathrm{Fac}}
\newcommand{\Pref}{\mathrm{Pref}}

\newcommand{\B}{\boldsymbol{\beta}}

\newcommand{\floor}[1]{\left\lfloor#1\right\rfloor}
\newcommand{\ceil}[1]{\left\lceil#1\right\rceil}
\newcommand{\DB}{d_{\boldsymbol{\beta}}}
\newcommand{\DA}{d_{\boldsymbol{\alpha}}}
\newcommand{\DBi}[1]{d_{{\B}^{(#1)}}}
\newcommand{\qDB}{d_{\boldsymbol{\beta}}^{*}}
\newcommand{\qDBi}[1]{d_{\boldsymbol{\beta}^{(#1)}}^{*}}
\newcommand{\Int}{[\![0,p-1]\!]}

\DeclareMathOperator{\Card}{Card}

\newcommand{\LB}{\ell_{\boldsymbol{\beta}}}
\newcommand{\LBi}[1]{\ell_{{\B}^{(#1)}}}
\newcommand{\qLB}{\ell_{\boldsymbol{\beta}}^{*}}
\newcommand{\qLBi}[1]{\ell_{\boldsymbol{\beta}^{(#1)}}^{*}}
\newcommand{\LA}{\ell_{\boldsymbol{\alpha}}}
\newcommand\restr[2]{{
  \left.\kern-\nulldelimiterspace 
  #1 
  \vphantom{\big|} 
  \right|_{#2} 
  }}

\title{Combinatorial properties of lazy expansions in Cantor real bases}
\author{Célia Cisternino}

\begin{document}

\begin{abstract}
The lazy algorithm for a real base $\beta$  is generalized to the setting of Cantor bases $\B=(\beta_n)_{n\in\N}$ introduced recently by Charlier and the author. 
To do so, let $x_{\B}$ be the greatest real number that has a $\B$-representation $a_0a_1a_2\cdots$ such that each letter $a_n$ belongs to $\{0,\ldots,\ceil{\beta_n}-1\}$. 
This paper is concerned with the combinatorial properties of the lazy $\B$-expansions, which are defined when $x_{\B}<+\infty$. 
As an illustration, Cantor bases following the Thue-Morse sequence are studied and a formula giving their corresponding value of $x_{\B}$ is proved.
First, it is shown that the lazy $\B$-expansions are obtained by ``flipping'' the digits of the greedy $\B$-expansions. 
Next, a Parry-like criterion characterizing the sequences of non-negative integers that are the lazy $\B$-expansions of some real number in $(x_{\B}-1,x_{\B}]$ is proved. 
Moreover, the lazy $\B$-shift is studied and in the particular case of alternate bases, that is the periodic Cantor bases, an analogue of Bertrand-Mathis' theorem in the lazy framework is proved: the lazy $\B$-shift is sofic if and only if all quasi-lazy $\B^{(i)}$-expansions  of $x_{\B^{(i)}}-1$ are ultimately periodic, where $\B^{(i)}$ is the $i$-th shift of the alternate base $\B$.
\end{abstract}

\maketitle

\bigskip
\hrule
\bigskip



\bigskip
\hrule
\bigskip


\section{Introduction}

Two well-known generalizations of the integer base representations are the Cantor and real base representations. The former was introduced by  Cantor in 1869~\cite{Cantor1869}. The Cantor representation of a real number $x$ via a base sequence $(b_n)_{n\in \N}\in (\N_{\ge2})^{\N}$ is an infinite sequence $a_0a_1a_2\cdots$ of non-negative integers such that
\[
x=\sum_{n\in\N}\frac{a_n}{\prod_{i=0}^n b_i}.
\]
The latter was defined by Rényi in 1957~\cite{Renyi1957} and well understood since the pioneering work of Parry in 1960~\cite{Parry1960}. A real base representation of a real number $x$ via a real base $\beta>1$ is an infinite sequence $a_0a_1a_2\cdots$ of non-negative integers such that 
\[
x=\sum_{n\in\N}\frac{a_n}{\beta^{i+1}}.
\]

Gathering both, the notion of Cantor real bases was introduced by Charlier and the author in a recent work~\cite{CharlierCisternino2021}. 
Note that these type of representations involving more than one base simultaneously and independently aroused the interest of mathematicians~\cite{CaalimaDemegillo2020, CharlierCisternino2021, 
Li2021, Neunhauserer2021, 
ZouKomornikLu2021}. 

A \emph{Cantor real base} is a sequence $\B=(\beta_n)_{n\in\N}$ of real numbers greater than $1$  such that $\prod_{n\in\N}\beta_n=+\infty$.
A representation of a real number $x$ via a Cantor real base $\B=(\beta_n)_{n\in\N}$ is an infinite sequence $a_0a_1a_2\cdots$ over $\N$ such that
\[
x=\sum_{n\in\N}\frac{a_n}{\prod_{i=0}^n \beta_i}.
\]
The digits of a $\B$-representation can be chosen by using several appropriate algorithms. As in the real base theory, in order to represent non-negative real numbers smaller than or equal to $x_{\B}$ where \[x_{\B}=\sum_{n\in\N}\frac{\ceil{\beta_n}-1}{\prod_{i=0}^n\beta_i},\] the most commonly used algorithms are the greedy and the lazy ones. 
In the greedy algorithm, each digit is chosen as the largest possible among $0,\ldots,\ceil{\beta_n}-1$ at position $n$. At the other extreme, the lazy algorithm picks the least possible digit at each step. The so-obtained $\B$-representations are respectively called the greedy and lazy $\B$-expansions. 

In the initial work~\cite{CharlierCisternino2021}, the combinatorial properties of the greedy $\B$-expansions of real numbers in $[0,1)$ were investigated. In particular, generalizations of several combinatorial results of real base expansions were obtained, such as Parry's criterion for greedy expansions and, while considering periodic Cantor real bases, called \emph{alternate bases}, Bertrand-Mathis' characterization of sofic shifts. 
Next, in~\cite{CharlierCisterninoDajani2021}, in the particular case of alternate bases, the lazy expansions were defined and both greedy and lazy expansions were studied in terms of dynamics. These results generalize the well-known ones from the theory of real base expansions (see~\cite{DajaniKraaikamp2002,DajaniKraaikamp2002-2,Parry1960,Renyi1957}). Note that the lazy real base expansions have been widely studied in terms of dynamics and, to the best of the author’s knowledge, not in terms of combinatorics. 

The goal of this paper is to study the combinatorial properties of the lazy expansions in Cantor real bases. 
In particular, the aim is to obtain a version of Parry's theorem~\cite{Parry1960} and Bertrand-Mathis' theorem~\cite{Bertrand-Mathis1989} in the lazy Cantor real base framework.

This paper is organized as follows. 
First, the Cantor bases and the associated greedy and lazy algorithms are introduced in Section~\ref{Sec : CantorRealBases}. Note that the lazy algorithm is defined when $x_{\B}<+\infty$ hence, this paper deals with Cantor bases such that $x_{\B}<+\infty$. As an illustration, in Section~\ref{Sec : CantorRealBases}, the value of $x_{\B}$ is studied when $\B$ is a Cantor base defined thanks to the Thue-Morse sequence over an alphabet $\{\alpha,\beta\}$ with $\alpha,\beta>1$, that is $\B=(\alpha,\beta,\beta,\alpha,\beta,\alpha,\alpha,\beta,\cdots)$.
Next, it is shown in Section~\ref{Sec : Flip} that the lazy $\B$-expansions are obtained by ``flipping'' the digits of the greedy $\B$-expansions. 
This allow us, to translate the greedy properties from~\cite{CharlierCisternino2021}  to their lazy analogues. 
Section~\ref{Sec : FirstPropertiesLazy} is then concerned by first few properties of lazy $\B$-expansions. 
Then, we define the quasi-lazy $\B$-expansions of $x_{\B}-1$ in Section~\ref{Sec :  QuasiGreedyLazy} and show that the same ``flip'' permits us to go from the quasi-greedy $\B$-expansion 
to the quasi-lazy one. 
Hence, in Section~\ref{Sec : Admissible}, the lazy $\B$-admissible sequences are studied and a Parry-like criterion characterizing the lazy $\B$-expansions is proved. Finally, in Section~\ref{Sec : LazyShift}, the lazy $\B$-shift is studied and in the particular case of alternate bases, an analogue of Bertrand-Mathis' theorem in the lazy case is proved. That is, if $\B$ is an alternate base, we obtain that the lazy $\B$-shift is sofic if and only if all quasi-lazy $\B^{(i)}$-expansions  of $x_{\B^{(i)}}-1$ are ultimately periodic, where $\B^{(i)}$ is the $i$-th shift of the alternate base $\B$.

\section{Cantor real bases}
\label{Sec : CantorRealBases}

In this section, the needed definitions and conventions are given. Throughout this text, if $a$ is an infinite word then for all $n\in\N$, $a_n$ designates its letter indexed by $n$, so that $a=a_0a_1a_2\cdots$, that is the $(n+1)^{\text{st}}$ letter of $a$. Moreover, an interval of non-negative integers $\{i, \ldots , j\}$ with\footnote{If $i>j$, we take the convention that $[\![i,j]\!]$ is the empty set.} $i\le j$ is denoted $[\![i,j]\!]$ and $\floor{\cdot}$ and $\ceil{\cdot}$ respectively denote the floor and ceiling functions.\\

A \emph{Cantor real base}, or simply a \emph{Cantor base}, is a sequence $\B=(\beta_n)_{n\in\N}$ of real numbers greater than $1$  such that $\prod_{n\in\N}\beta_n=+\infty$.
For instance, any sequence $\B=(\beta_n)_{n\in\N}$ of real numbers greater than $1$ that takes only finitely many values is a Cantor base since in this case, the condition $\prod_{n\in\N}\beta_n=+\infty$ is trivially satisfied. 
In particular, if $\beta>1$ then $\B=(\beta,\beta,\ldots)$ is a Cantor base and in this case, all notions coincides with the widely studied theory of $\beta$-expansions. 
An \emph{alternate base} is a periodic Cantor base, that is a Cantor base for which there exists $p\in\N_{\ge 1}$ such that for all $n\in\N$, $\beta_n=\beta_{n+p}$. In this case, we simply write $\B=(\overline{\beta_0,\ldots,\beta_{p-1}})$ and the integer $p$ is called the \emph{length} of the alternate base $\B$. 

Let $\B$ be a Cantor base. We define
\[
	\B^{(n)}=(\beta_n,\beta_{n+1},\ldots)\quad \text{for all }n\in\N.
\]
In particular $\B^{(0)}=\B$.
The \emph{$\B$-value (partial) map} $\val_{\B}\colon (\R_{\ge 0})^\N\to \R_{\ge 0}$ by
\begin{equation}
\label{eq:representationCantor}
\val_{\B}(a)=\sum_{n\in\N}\frac{a_n}{\prod_{i=0}^n\beta_i}
\end{equation}
for any infinite word $a$ over $\R_{\ge 0}$, provided that the series converges. A \emph{$\B$-representation} of a non-negative real number $x$ is an infinite word $a\in\N^\N$ such that $\val_{\B}(a)=x$.
A $\B$-representation is said to be \emph{finite} if it ends with infinitely many zeros, and \emph{infinite} otherwise. 
The \emph{length} of a finite $\B$-representation is the length of the longest prefix ending in a non-zero digit. 
When a $\B$-representation is finite, we sometimes omit to write the tail of zeros.

\subsection{Greedy algorithm on $\boldsymbol{[0,1)}$}\label{Sec : Greedy}

For $x\in [0,1)$, a distinguished $\B$-representation \linebreak$\varepsilon_{\B,0}(x)\varepsilon_{\B,1}(x)\varepsilon_{\B,2}(x)\cdots$, called the \emph{greedy $\B$-expansion} is obtained from the \emph{greedy algorithm}. If the first $N$ digits of the greedy $\B$-expansion of $x$ are given by $\varepsilon_{\B,0}(x),\ldots,\varepsilon_{\B,N-1}(x)$, then the next digit $\varepsilon_{\B,N}(x)$ is the greatest integer such that 
\[
	\sum_{n=0}^N	\frac{\varepsilon_{\B,n}(x)}{\prod_{k=0}^n\beta_k}\le x.
\]
In particular, for all $n\in \N$, the digit $\varepsilon_{\B,n}(x)$ belongs to the alphabet $[\![0,\ceil{\beta_n}-1]\!]$.
The greedy algorithm can be equivalently defined as follows:
\begin{itemize}
\item $\varepsilon_{\B,0}(x)=\floor{\beta_0 x}$ and $r_{\B,0}(x)=\beta_0 x-\varepsilon_{\B,0}(x)$
\item $\varepsilon_{\B,n}(x)=\floor{\beta_n r_{\B,n-1}(x)}$ and $r_{\B,n}(x)=\beta_n r_{\B,n-1}(x)-\varepsilon_{\B,n}(x)$ for $n\in\N_{\ge 1}$.
\end{itemize}
The obtained $\B$-representation is denoted by $\DB(x)$ and is called the \emph{greedy $\B$-expansion} of $x$. When the context is clear, we simply denote $\varepsilon_{\B,n}(x)$ by $\varepsilon_{n}(x)$ and $r_{\B,n}(x)$ by $r_{n}(x)$.

\begin{example}\label{Ex : AlternateBase}
Consider the alternate base $\B=(\overline{\frac{1+ \sqrt{13}}{2},\frac{5+ \sqrt{13}}{6}})$ already studied in~\cite{CharlierCisternino2021} and~\cite{CharlierCisterninoDajani2021}. We have $d_{\B}(\frac{-5+2\sqrt{13}}{3})=11$ and $d_{\B}(\frac{2+\sqrt{13}}{9})=(10)^\omega$ where the $\omega$ notation means an infinite repetition.
\end{example}

\begin{example}\label{Ex : GreedyTM}
Let $\alpha=\frac{1+ \sqrt{13}}{2}$ and $\beta=\frac{5+ \sqrt{13}}{6}$ and consider the Cantor base $\B=(\beta_n)_{n\in\N}$ from~\cite{CharlierCisternino2021} defined by 
\begin{align}
\label{Eq : CantorTM}
\beta_n=\begin{cases} 
\alpha & \text{if } | \rep_2(n) |_1 \equiv 0 \pmod 2\\
\beta & \text{otherwise}
\end{cases}
\end{align}
for all $n\in \N$, where $\rep_2$ is the function mapping any non-negative integer to its $2$-expansion and $|u|_1$ is the number of occurrences of the letter $1$ in the word $u$. We get $\B=(\alpha,\beta,\beta,\alpha,\beta,\alpha,\alpha,\beta, \ldots)$ where the infinite word $\beta_0\beta_1\beta_2\cdots$ is the Thue-Morse word over the alphabet $\{\alpha,\beta\}$. 
The greedy $\B$-expansion of $\frac{1}{2}$ has $10001$ as a prefix and $d_{\B}(\frac{65-18\sqrt{13}}{6})=1002$.
\end{example}

\subsection{Lazy algorithm on $\boldsymbol{(x_{\B}-1,x_{\B}]}$}\label{Sec : Lazy}

Considering a Cantor base $\B$, define
\begin{equation*}
\label{Eq : x_B}
	x_{\B}=\sum_{n\in\N}\frac{\ceil{\beta_n}-1}{\prod_{k=0}^n\beta_k}.
\end{equation*}
Either this series converges or $x_{\B}=+\infty$.
In both cases, this corresponds to the greatest real number that has a $\B$-representation $a_0a_1a_2\cdots$ such that for all $n\in \N$ the letter $a_n$ belongs to the alphabet $[\![0,\ceil{\beta_n}-1]\!]$. 

Since the greedy algorithm converges on $[0,1)$, it can be easily seen that $x_{\B}\ge 1$. Moreover, for all $n\in\N$, 
\begin{equation}
\label{Eq : EqualitiesXB}
	x_{\B^{(n)}}=\frac{x_{\B^{(n+1)}}+\ceil{\beta_n}-1}{\beta_n}.
\end{equation}
Hence, it can be easily proved that $x_{\B}=1$ if and only if $\beta_n\in \N_{\ge2}$ for all $n\in \N$.

\begin{example}\label{Ex : XBAlternateBase}
Consider the alternate base $\B=(\overline{\frac{1+ \sqrt{13}}{2},\frac{5+ \sqrt{13}}{6}})$ from Example~\ref{Ex : AlternateBase}. We get $x_{\B}=\frac{5+7\sqrt{13}}{18}\simeq 1.67$ and $x_{\B^{(1)}}=\frac{2+\sqrt{13}}{3}\simeq 1.86$.
\end{example}

\begin{example}\label{Ex : XBTM}
Let $\alpha,\beta>1$ and let $\B=(\alpha,\beta,\beta,\alpha,\beta,\alpha,\alpha,\beta, \ldots)$ be the Thue-Morse Cantor base on $\{\alpha,\beta\}$ defined as~\eqref{Eq : CantorTM}. 
For all $n\ge 1 $, let 
\[
x_n=\sum_{m=0}^{2^n-1}\frac{\ceil{\beta_m}-1}{\prod_{k=0}^{m}\beta_k}.
\]
We get $x_{\B}=\lim_{n\to +\infty}x_n$.
Similarly, let $\overline{\B}$ denote the Cantor base $\overline{\B}=(\overline{\beta_n})_{n\in\N}$ where $\overline{\alpha}=\beta$ and $\overline{\beta}=\alpha$.
We get $\overline{\B}=(\beta,\alpha,\alpha,\beta,\alpha,\beta,\beta,\alpha, \ldots)$. For all $n\ge 1$, denote 
\[
y_n=\sum_{m=0}^{2^n-1}\frac{\big\lceil\ \overline{\beta_m}\ \big\rceil-1}{\prod_{k=0}^{m}\overline{\beta_k}}.
\]
By definition of the Thue-Morse sequence, for all $n\in \N$ we have 
\[
(\beta_{2^n},\beta_{2^n+1},\ldots, \beta_{2^{n+1}-1})=(\overline{\beta_0}, \overline{\beta_1},\ldots,\overline{\beta_{2^n-1}}).
\]
Moreover, for all $n\ge 1$ the sequence $(\beta_0,\ldots,\beta_{2^n-1})$ has the same number of $\alpha$ and $\beta$. We get $\prod_{k=0}^{2^n-1}\beta_k=(\alpha\beta)^{2^{n-1}}$.
Hence, we have
\begin{align*}
\begin{cases}
&x_1=\frac{\ceil{\alpha}-1}{\alpha}+\frac{\ceil{\beta}-1}{\alpha\beta}, \\
& y_1=\frac{\ceil{\beta}-1}{\beta}+\frac{\ceil{\alpha}-1}{\beta\alpha},\\
& x_{n+1}=x_n+\frac{1}{(\alpha\beta)^{2^{n-1}}} y_n, \quad \forall n\ge 1\\
& y_{n+1}=y_n+\frac{1}{(\alpha\beta)^{2^{n-1}}} x_n, \quad \forall n\ge 1.
\end{cases}
\end{align*}
That is, for all $n \ge 1$, we have 
\[
v_{n+1}=A_nv_n 
\]
where,
\[v_n=\begin{pmatrix}
x_n\\
y_n
\end{pmatrix} \quad \text{ and } \quad A_n=\begin{pmatrix}
1 & \tfrac{1}{(\alpha\beta)^{2^{n-1}}}\\
\tfrac{1}{(\alpha\beta)^{2^{n-1}}} & 1
\end{pmatrix}.
\]
For all $n\ge 1$, the eigenvalues of the matrix $A_n $ are $1+\tfrac{1}{(\alpha\beta)^{2^{n-1}}}$ and $1-\tfrac{1}{(\alpha\beta)^{2^{n-1}}}$ of eigenvectors $\begin{pmatrix}
1\\
1
\end{pmatrix}$ and $\begin{pmatrix}
1\\
-1
\end{pmatrix}$
respectively.
Moreover, we have 
\[
v_1= 
\tfrac{x_1+y_1}{2} \begin{pmatrix}
1\\
1
\end{pmatrix}
+ 
\tfrac{x_1-y_1}{2} 
\begin{pmatrix}
1\\
-1
\end{pmatrix}.
\]
We obtain
\begin{align*}
v_{n+1}&=A_nA_{n-1}\cdots A_1v_1\\
&= \tfrac{x_1+y_1}{2}A_nA_{n-1}\cdots A_1\begin{pmatrix}
1\\
1
\end{pmatrix} + \tfrac{x_1-y_1}{2} A_nA_{n-1}\cdots A_1 \begin{pmatrix}
1\\
-1
\end{pmatrix}\\ 
&= \tfrac{x_1+y_1}{2}\prod_{k=1}^n \big(1+\tfrac{1}{(\alpha\beta)^{2^{k-1}}}\big)\begin{pmatrix}
1\\
1
\end{pmatrix} + \tfrac{x_1-y_1}{2} \prod_{k=1}^n \big(1-\tfrac{1}{(\alpha\beta)^{2^{k-1}}}\big) \begin{pmatrix}
1\\
-1
\end{pmatrix}. 
\end{align*}
Then, the value of $x_{\B}$ can be computed by
\begin{align*}
x_{\B}=\lim_{n\to +\infty}x_n=\tfrac{x_1+y_1}{2}\prod_{k\in \N_{\ge1}} \big(1+\tfrac{1}{(\alpha\beta)^{2^{n-1}}}\big) + \tfrac{x_1-y_1}{2} \prod_{k\in \N_{\ge1}} \big(1-\tfrac{1}{(\alpha\beta)^{2^{n-1}}}\big).
\end{align*}
We now study the two infinite products in the above formula. We have 
\begin{align*}
&\Big(\prod_{k\in \N_{\ge1}} \big(1+\tfrac{1}{(\alpha\beta)^{2^{k-1}}}\big)\Big)
\Big(\prod_{k\in \N_{\ge1}} \big(1-\tfrac{1}{(\alpha\beta)^{2^{k-1}}}\big)\Big)
\\
= & \prod_{k\in \N_{\ge1}}^ \Big(\big(1+\tfrac{1}{(\alpha\beta)^{2^{k-1}}}\big) \big(1-\tfrac{1}{(\alpha\beta)^{2^{k-1}}}\big)\Big)\\
=& \prod_{k=2}^\infty \big(1-\tfrac{1}{(\alpha\beta)^{2^{k-1}}}\big).
\end{align*}
Hence, we get
\[
\prod_{k\in \N_{\ge1}} \big(1+\tfrac{1}{(\alpha\beta)^{2^{k-1}}}\big)= \frac{1}{1-\frac{1}{\alpha\beta}}.
\]
Moreover, consider the function $f$ defined by $f(z)=\sum_{m\in \N} (-1)^{t_m}z^m$ where $t_0t_1t_2\cdots$ is the Thue-Morse sequence over the alphabet $\{0,1\}$. 
By the infinite product definition of the Thue-Morse sequence, we get
\begin{align*}
\prod_{k\in \N_{\ge1}} \big(1-\tfrac{1}{(\alpha\beta)^{2^{k-1}}}\big)=f(\tfrac{1}{\alpha\beta}).
\end{align*}
Then, the value of $x_{\B}$ can be computed by
\begin{align*}
x_{\B}=
\tfrac{x_1+y_1}{2}\big(\frac{1}{1-\frac{1}{\alpha\beta}}\big)+ \tfrac{x_1-y_1}{2} f(\tfrac{1}{\alpha\beta}).
\end{align*}
In particular, by considering the Cantor base from Example~\ref{Ex : GreedyTM}, a computer approximation of $f(\tfrac{1}{\alpha\beta})$ gives $0.627941$. Hence, we get $x_{\B}\simeq 1.73295$.
\end{example}

\begin{example}
Consider the Cantor base $\B=(1+\frac{1}{n+1})_{n\in \N}$. For all $n\in \N$, we have $\beta_n=\frac{n+2}{n+1}$ so we get
\[
x_{\B}=\sum_{n\in\N}\frac{1}{\prod_{k=0}^n\frac{k+2}{k+1}}=\sum_{n\in\N}\frac{1}{n+2}=+\infty.
\]
\end{example}

As said in~\cite[Section 3]{CharlierCisterninoDajani2021}, if $x_{\B}<+\infty$, the other extreme $\B$-expansions of real number, namely the \emph{lazy $\B$-expansions}, is defined.
Hence, from now on, consider a Cantor base $\B=(\beta_n)_{n\in \N}$ such that $x_{\B}<+\infty$. For instance, any Cantor base $\B$ that takes only finitely many values has finite corresponding $x_{\B}$. 

In the greedy algorithm, each digit is chosen as the largest possible at the considered position. 
On the contrary, in the lazy algorithm, each digit is chosen as the least possible at each step. 
The \emph{lazy algorithm} is defined as follows: for $x\in (x_{\B}-1,x_{\B}]$, if the first $N$ digits of the lazy $\B$-expansion of $x$ are given by $\xi_{\B,0}(x),\ldots,\xi_{\B,N-1}(x)$, then the next digit $\xi_{\B,N}(x)$ is the least element in $[\![0,\ceil{\beta_N}-1]\!]$ such that 
\[
	\sum_{n=0}^N\frac{\xi_{\B,n}(x)}{\prod_{k=0}^n\beta_k}
	+\sum_{n=N+1}^{+\infty}\frac{\ceil{\beta_n}-1}{\prod_{k=0}^n\beta_k}
	\ge x.
\]
The lazy algorithm can be equivalently defined as follows:
\begin{itemize}
\item $\xi_{\B,0}(x)=\lceil\beta_0 x- x_{\B^{(1)}}\rceil$ and $s_{\B,0}(x)=\beta_0 x-\xi_{\B,0}(x)$
\item $\xi_{\B,n}(x)=\lceil\beta_n s_{\B,n-1}(x)-x_{\B^{(n+1)}}\rceil$ and $s_{\B,n}(x)=\beta_n s_{\B,n-1}(x)-\xi_{\B,n}(x)$ for $n\in\N_{\ge 1}$.
\end{itemize}
The obtained $\B$-representation of $x\in (x_{\B}-1,x_{\B}]$ is denoted by $\LB(x)$ and is called the \emph{lazy $\B$-expansion} of $x$. As before, if the context is clear, the indexes $\B$ in the writings $\xi_{\B,n}(x)$ and $s_{\B,n}(x)$ are omitted.

\begin{example}\label{Ex : LazyAlternateBase35}
We continue Examples~\ref{Ex : AlternateBase} and~\ref{Ex : XBAlternateBase}. The first $5$ digits of $\LB(\frac{35-5\sqrt{13}}{18})$ are $10212$.
\end{example}

\section{Flip greedy and get lazy}\label{Sec : Flip}

In~\cite[Section 5]{CharlierCisterninoDajani2021}, in the alternate base framework, both greedy and lazy expansions were compared. The following result generalizes this comparison to the Cantor base expansions. 

For a Cantor base $\B=(\beta_n)_{n\in \N}$, we let $A_{\B}$ denote the (possibly infinite) alphabet $[\![0,\sup_{n\in\N}(\ceil{\beta_n}-1)]\!]$. Note that, if the supremum is infinite, the alphabet $A_{\B}$ is made of all non-negative integers. Any greedy and lazy $\B$-expansion belongs to $A_{\B}^{\N}$ and more precisely to the set of infinite words $a\in A_{\B}^{\N}$ such that, for all $n\in \N$, the letter $a_n$ belongs to $[\![0 , \ceil{\beta_n}-1]\!]$. From now on, let $\prod_{n\in \N}[\![0 , \ceil{\beta_n}-1]\!]$ denote this set of infinite words.\\

Let $\theta_{\B}$ be the map defined by
\begin{align*}
\theta_{\B} \colon &\prod_{n\in \N}[\![0 , \ceil{\beta_n}-1]\!]\to \prod_{n\in \N}[\![0 , \ceil{\beta_n}-1]\!], \\ &a_0a_1\cdots\mapsto (\ceil{\beta_0}-1-a_0)(\ceil{\beta_1}-1-a_1)\cdots.
\end{align*}
The map $\theta_{\B}$ is continuous with respect to the topology induced by the prefix distance, bijective and the inverse map $\theta_{\B}^{-1}$ is the map $\theta_{\B}$ itself. 
For any infinite word $a\in \prod_{n\in \N}[\![0 , \ceil{\beta_n}-1]\!]$, we get 
\begin{equation}\label{Eq : ValTheta}
\val_{\B}(\theta_{\B}(a))=x_{\B}-\val_{\B}(a).
\end{equation}
Moreover, the map $\theta_{\B}$ is decreasing with respect to the lexicographic order, that is, for all infinite words $a$ and $b$ in $\prod_{n\in \N}[\![0 , \ceil{\beta_n}-1]\!]$, we get
\begin{equation}\label{Eq : ThetaDecreasing}
a<_{\lex} b \iff  \theta_{\B}(a)>_{\lex}\theta_{\B}(b).
\end{equation}
The map $\theta_{\B}$ is the key of the reasoning of this paper. In fact, as shown in the following result, it will allow us to ``flip'' the greedy expansions in order to get the lazy ones.

\begin{proposition}\label{Pro : LinkGreedyLazy}
For all $x\in [0,1)$ and all $n\in \N$, we have
$\xi_{\B,n}(x_{\B}-x)=\ceil{\beta_n}-1-\varepsilon_{\B,n}(x)$ and $s_{\B,n}(x_{\B}-x)=x_{\B^{(n+1)}}-r_{\B,n}(x)$.
In particular, we get $\LB(x_{\B}-x)=\theta_{\B}(\DB(x))$.
\end{proposition}
\begin{proof}
Consider $x\in [0,1)$. We proceed by induction on $n$. By~\eqref{Eq : EqualitiesXB}, we have 
\begin{align*}
\xi_{\B,0}(x_{\B}-x)=&\lceil\beta_0(x_{\B}-x)-x_{\B^{(1)}}\rceil\\
=&\ceil{\ceil{\beta_0}-1-\beta_0x}\\
=&\ceil{\beta_0}-1+\ceil{-\beta_0x}\\
=&\ceil{\beta_0}-1-\floor{\beta_0x}\\
=&\ceil{\beta_0}-1-\varepsilon_{\B,0}(x).
\end{align*}
Moreover, we get 
\begin{align*}
s_{\B,0}(x_{\B}-x)=&\beta_0(x_{\B}-x)-(\ceil{\beta_0}-1-\varepsilon_{\B,0}(x))\\
=& \beta_0x_{\B}-(\ceil{\beta_0}-1)-( \beta_0 x - \varepsilon_{\B,0}(x))\\
=& x_{\B^{(1)}}-r_{\B,0}(x)
\end{align*}
where~\eqref{Eq : EqualitiesXB} is used again in the last equality. By induction, for all $n\in \N_{\ge 1}$, we have 
\begin{align*}
\xi_{\B,n}(x_{\B}-x)=&\lceil\beta_n s_{\B,n-1}(x_{\B}-x)-x_{\B^{(n+1)}}\rceil\\
=& \lceil\beta_n (x_{\B^{(n)}}- r_{\B,n-1}(x_{\B}-x))-x_{\B^{(n+1)}}\rceil\\
=&\ceil{\ceil{\beta_n}-1-\beta_n r_{\B,n-1}(x_{\B}-x)}\\
=&\ceil{\beta_n}-1-\floor{\beta_n r_{\B,n-1}(x_{\B}-x)}\\
=&\ceil{\beta_n}-1-\varepsilon_{\B,n}(x)
\end{align*}
and 
\begin{align*}
s_{\B,n}(x_{\B}-x)=&\beta_n s_{\B,n-1}(x_{\B}-x)-\xi_{\B,n}(x_{\B}-x)\\
=& \beta_n( x_{\B^{(n)}}-r_{\B,n-1}(x))-( \ceil{\beta_n}-1-\varepsilon_{\B,n}(x) )\\
=& x_{\B^{(n+1)}}-r_{\B,n}(x).
\end{align*}
In particular, we can conclude that $\LB(x_{\B}-x)=\theta_{\B}(\DB(x))$.
\end{proof}

\begin{example}
Let $\B=(\overline{\frac{1+ \sqrt{13}}{2},\frac{5+ \sqrt{13}}{6}})$ be the alternate base considered in Example~\ref{Ex : AlternateBase}.
By Proposition~\ref{Pro : LinkGreedyLazy}, the lazy $\B$-expansion of $x_{\B}-\frac{-5+2\sqrt{13}}{3}=\frac{25-5\sqrt{13}}{18}$ equals $10(21)^\omega$ 
since $\DB(\frac{-5+2\sqrt{13}}{3})=11$. This coincides with Example~\ref{Ex : LazyAlternateBase35}.
\end{example}

\begin{example}\label{Ex : GreedyTM1XB1}
We continue Example~\ref{Ex : GreedyTM}. The lazy $\B
$-expansion of $x_{\B}-\frac12\simeq 1.23$ has $11120$ as a prefix.
\end{example}

Thanks to Proposition~\ref{Pro : LinkGreedyLazy}, in the sequel, results from~\cite{CharlierCisternino2021} on greedy $\B$-expansions will be translated in terms of lazy $\B$-expansions. 
The differences between the greedy and lazy $\B$-expansions will be highlighted in the text.

\section{First properties of lazy expansions}\label{Sec : FirstPropertiesLazy}

For any alphabet $A$, the \emph{shift operator over $A$}, denoted by $\sigma_A$, is defined by 
\[
	\sigma_A\colon A^\N\to A^\N,\ a_0a_1a_2\cdots\mapsto a_1a_2a_3\cdots.
\] 
Throughout the text, whenever there is no ambiguity on the alphabet, we simply write $\sigma$ instead of $\sigma_{A_{\B}}$.

\begin{lemma}\label{Lem : CommuteShiftTheta}
For all $n\in \N$, we have $\sigma^n \circ \theta_{\B}= \theta_{\B^{(n)}}\circ \sigma^n$ on $\prod_{n\in \N}[\![0 , \ceil{\beta_n}-1]\!]$.
\end{lemma}
\begin{proof}
This is a straightforward verification.
\end{proof}

\begin{proposition}\label{Pro : CommuteShiftThetaLB}
For all $x \in  (x_{\B}-1,x_{\B}]$ and all $n\in \N$, we have 
\[
\sigma^n(\LB(x))=\LBi{n}(s_{\B,n-1}(x)).
\]
\end{proposition}
\begin{proof}
This is a consequence of Proposition~\ref{Pro : LinkGreedyLazy}, Lemma~\ref{Lem : CommuteShiftTheta} and~\cite[Proposition 8]{CharlierCisternino2021} since for all $x \in  (x_{\B}-1,x_{\B}]$ we have 
\begin{align*}
\sigma^n(\LB(x))&= \sigma^n\circ  \theta_{\B} (d_{\B}(x_{\B}-x))\\
&= \theta_{\B^{(n)}}\circ  \sigma^n ( d_{\B}(x_{\B}-x))\\
&= \theta_{\B^{(n)}}( d_{\B^{(n)}}(r_{\B,n-1}(x_{\B}-x)))\\
&= \LBi{n}(x_{\B^{(n)}}-r_{\B,n-1}(x_{\B}-x))\\
&= \LBi{n}(s_{\B,n-1}(x)).\qedhere
\end{align*}
\end{proof}

\begin{proposition}
\label{Pro : LazyLess}
Let $a$ be an infinite word over $\N$ and $x\in (x_{\B}-1,x_{\B}]$.
We have $a=\LB(x)$ if and only if $a \in \prod_{n\in \N}[\![0 , \ceil{\beta_n}-1]\!]$, $\val_{\B}(a)=x$ and for all $\ell\in\N$,
\begin{equation}
\sum_{n=\ell+1}^{+\infty}\frac{a_n}{\prod_{k=0}^n{\beta_k}} >  \frac{x_{\B^{(\ell+1)}}-1}{\prod_{k=0}^{\ell}{\beta_k}}.\nonumber
\end{equation}
\end{proposition}
\begin{proof}
Consider $a\in \N^{\N}$ and $x\in (x_{\B}-1,x_{\B}]$. By Proposition~\ref{Pro : LinkGreedyLazy}, we have $a=\LB(x)$ if and only if $a \in \prod_{n\in \N}[\![0 , \ceil{\beta_n}-1]\!]$ and $\theta_{\B}(a)=\DB(x_{\B}-x)$. By~\cite[Lemma 9]{CharlierCisternino2021}, we get $a=\LB(x)$ if and only if $a \in \prod_{n\in \N}[\![0 , \ceil{\beta_n}-1]\!]$, $\val_{\B}(\theta_{\B}(a))=x_{\B}-x$ and for all $N\in\N$,
\begin{equation*}
\sum_{n=N+1}^{+\infty}\frac{\ceil{\beta_n}-1-a_n}{\prod_{k=0}^n{\beta_k}} < \frac{1}{\prod_{k=0}^{N}{\beta_k}}.
\end{equation*}
We conclude the proof by~\eqref{Eq : ValTheta} and by definition of $x_{\B^{(N+1)}}$.
\end{proof}

\begin{proposition}
\label{Pro : LazyLexMin}
The lazy $\B$-expansion of a real number $x\in (x_{\B}-1,x_{\B}]$ is lexicographically minimal among all $\B$-representations of $x$ in $\prod_{n\in \N}[\![0 , \ceil{\beta_n}-1]\!]$.
\end{proposition}
\begin{proof}
Let $x\in (x_{\B}-1,x_{\B}]$ and let $a\in \prod_{n\in \N}[\![0 , \ceil{\beta_n}-1]\!]$ be a $\B$-representation of $x$. Suppose that $a<_{\lex} \LB(x)$. 
By~\eqref{Eq : ThetaDecreasing}, we get $\theta_{\B}(a)>_{\lex} \theta_{\B}(\LB(x))$.  By~\eqref{Eq : ValTheta}, $\theta_{\B}(a)$ is a $\B$-representation of $x_{\B}-x$. Moreover, by Proposition~\ref{Pro : LinkGreedyLazy} and since the inverse map $\theta_{\B}^{-1}$ is the map $\theta_{\B}$ itself, we have $\theta_{\B}(\LB(x))=\DB(x_{\B}-x)$. This is absurd since, by~\cite[Proposition 12]{CharlierCisternino2021}, $\DB(x_{\B}-x)$ is lexicographically maximal among all $\B$-representations of $x_{\B}-x$. 
\end{proof}

Note that, contrary to~\cite[Proposition 12]{CharlierCisternino2021}, it cannot be stated that ``the lazy $\B$-expansion of a real number $x\in (x_{\B}-1,x_{\B}]$ is lexicographically minimal among all $\B$-representations of $x$''. In fact, the alphabet of the $\B$-representations of $x$ must be fixed as shown in the following example.

\begin{example}\label{Ex : AlphabetImportant}
Let $\B$ be the alternate base from Example~\ref{Ex : AlternateBase} and consider $x=8-2\sqrt{13}$. We have $x\in (x_{\B}-1,x_{\B}]$ and the lazy $\B$-expansion of $x$ has $01$ as a prefix. However, the infinite word $003330^\omega$ is a $\B$-representation of $x$ and $003330^\omega<_{\lex}\LB(x)$. This does not contradict Proposition~\ref{Pro : LazyLexMin} since the infinite word $003330^\omega$ does not belong to $\prod_{n\in \N}[\![0 , \ceil{\beta_n}-1]\!]$. 
\end{example}

\begin{proposition}
\label{pro:Increasing}
The function $\LB\colon (x_{\B}-1,x_{\B}]\to {A_{\B}}^\N$ is increasing: 
\[
	\forall x,y \in (x_{\B}-1,x_{\B}],\quad x<y \iff \LB(x) <_{\lex} \LB(y).
\]
\end{proposition}
\begin{proof}
Consider $x,y \in (x_{\B}-1,x_{\B}]$. By~\cite[Proposition 13]{CharlierCisternino2021}, Proposition~\ref{Pro : LinkGreedyLazy} and~\eqref{Eq : ThetaDecreasing}, we have 
\begin{align*}
x<y \iff & x_{\B}-x>x_{\B}-y\\
 \iff & \DB(x_{\B}-x)>_{\lex} \DB(x_{\B}-y)\\
 \iff & \theta_{\B}(\LB(x))>_{\lex} \theta_{\B}(\LB(y))\\
\iff & \LB(x)<_{\lex} \LB(y).\qedhere
\end{align*}
\end{proof}

\begin{remark}
Considering two Cantor bases $\boldsymbol{\alpha}=(\alpha_n)_{n\in\N}$ and $\B=(\beta_n)_{n\in\N}$ such that for all $n\in\N$, $\prod_{i=0}^n\alpha_i \le \prod_{i=0}^n\beta_i$, by~\cite[Proposition 15]{CharlierCisternino2021}, we have $\DA(x)\le_{\lex} \DB(x)$ for all $x\in [0,1)$.
However, an analogous result cannot be obtained for the lazy expansions. 
In fact, since the interval of definition of the lazy expansions depends on the considered Cantor base, it is not possible to state a result of the form ``for all $x\in I$, we have  $\LA(x)\le_{\lex} \LB(x)$ (or  $\LA(x)\ge_{\lex} \LB(x)$)'' where $I$ is a fixed interval.
Moreover, it is neither correct to say ``for all $x\in [0,1)$, we have  $\LA(x_{\boldsymbol{\alpha}}-x)\le_{\lex} \LB(x_{\B}-x)$ (or  $\LA(x_{\boldsymbol{\alpha}}-x)\ge_{\lex} \LB(x_{\boldsymbol{\alpha}}-x)$)''.
Indeed, this can already be seen while considering real bases, that is $\B=(\beta,\beta,\ldots)$ with $\beta>1$, as illustrated in Figure~\ref{Fig : RealBasesLazyNotOrdered} (where the notation $\beta$, $x_{\beta}$ and $\ell_{\beta}(\cdot)$ are used instead of $\B$, $x_{\B}$ and $\ell_{\B}(\cdot)$). 
\begin{figure}
\[
	\begin{array}{c|ccc}
	\beta & 2 & \frac{11}{5} & \frac{5}{2} \\[3pt]
	\hline
	x_\beta & \tfrac{\ceil{2}-1}{2-1}=1 & \tfrac{\ceil{\frac{11}{5}}-1}{\frac{11}{5}-1}=\frac53 & \tfrac{\ceil{\frac{5}{2}}-1}{\frac{5}{2}-1}= \frac43\\
	\ell_{\beta}(x_\beta-\frac{1}{2}) & 01^\omega & 1221\cdots & 1211\cdots 
	\end{array}
\]
\caption{Some lazy $\B$-expansions when $\B=(\beta,\beta,\ldots)$ with $\beta>1$.}
\label{Fig : RealBasesLazyNotOrdered}
\end{figure}
\end{remark}

\begin{remark}
Note that, some results as Propositions~\ref{Pro : CommuteShiftThetaLB} and~\ref{pro:Increasing} could also have been proved easily without any prerequisite from~\cite{CharlierCisternino2021}. In this paper, a choice has been made, that is to use as much as possible Proposition~\ref{Pro : LinkGreedyLazy} and results from~\cite{CharlierCisternino2021}.
\end{remark}

\section{Quasi-lazy expansions}\label{Sec :  QuasiGreedyLazy}

In this section, we define the quasi-lazy $\B$-expansion of $x_{\B}-1$ in order to obtain an analogous of Parry's theorem~\cite{Parry1960} characterizing the lazy expansions of real numbers in $(x_{\B}-1,x_{\B}]$.

First, let us define the \emph{quasi-greedy $\B$-expansion of $1$} by
\begin{equation}
\label{Eq : QuasiGreedy}
\qDB(1)=\lim_{x\to 1^{-}} \DB(x)
\end{equation}
where the limit is taken with respect to the prefix distance of infinite words.
Note that this limit exists by left continuity of $\DB$ in the neighborhood of $1$. 

\begin{remark}
The quasi-greedy $\B$-expansion of $1$ obtained in~\eqref{Eq : QuasiGreedy} coincides with the one defined in~\cite{CharlierCisternino2021}. In fact, let $t_0t_1\cdots$ denote the quasi-greedy $\qDB(1)$ from~\cite{CharlierCisternino2021}. By~\cite[Theorem 26 and Corollary 36]{CharlierCisternino2021}, for all $n\in \N$, the word $t_0\cdots t_n0^\omega$ is the greedy $\B$-expansion of a real number in $x_n\in[0,1)$. We have $\lim_{n\to +\infty}x_n=1$ and $\lim_{n\to +\infty} \DB(x_n)=\qDB(1)$. 
Hence, in what follows, the results from~\cite{CharlierCisternino2021} in terms of $\qDB(1)$ can be used.

Note that, in~\cite{CharlierCisternino2021}, we made a choice of definition for the greedy $\B$-expansion of $1$ and defined $\qDB(1)$ accordingly. However, in this paper I decided not to define the greedy $\B$-expansion of $1$. In fact, if this were the case, one would have expected to define the lazy $\B$-expansion of $x_{\B} -1$ analogously. This would have been done by extending the lazy algorithm over $x_{\B} -1 $ as in~\cite{CharlierCisterninoDajani2021}. 
However, in that case, $\LB(x_{\B}-1)$ would have not been the image of $\DB(1)$, chosen as in~\cite{CharlierCisternino2021},  by the map $\theta_{\B}$ when $\beta_0$ is an integer since if $\beta_0\in \N_{\ge 2}$, we have $\DB(1)=\beta_0 0^\omega$ whereas the first letter of $\LB(x_{\B}-1)$ is $0$.
\end{remark}

In order to get similar results from~\cite{CharlierCisternino2021} for lazy expansions, we define the \emph{quasi-lazy $\B$-expansion} of $x_{\B}-1$ as follows: 
\begin{equation}
\label{Eq : QuasiLazy}
\qLB(x_{\B}-1)=\lim_{x \to (x_{\B}-1)^+} \LB(x).
\end{equation}
Again, this limit exists by right continuity of $\LB$ in the neighborhood of $x_{\B}-1$. 
Let us first prove that, similarly to Proposition~\ref{Pro : LinkGreedyLazy}, the ``flip'' of the quasi-greedy $\B$-expansions of $1$ is the quasi-lazy $\B$-expansion of $x_{\B}-1$.

\begin{proposition}\label{Pro : LinkQuasiGreedyQuasiLazy}
We have $\qLB(x_{\B}-1)=\theta_{\B}(\qDB(1))$.
\end{proposition}
\begin{proof}
Consider a sequence of real numbers $(x_n)_{n\in \N}\in [0,1)^{\N}$ such that $\lim_{n\to +\infty}x_n=1$.
We have $(x_{\B}-x_n)_{n\in \N}\in (x_{\B}-1,x_{\B}]^{\N}$ and $\lim_{n\to +\infty}(x_{\B}-x_n)=x_{\B}-1$. Hence, by continuity of $\theta_{\B}$ and by Proposition~\ref{Pro : LinkGreedyLazy}, we get
\begin{align*}
\theta_{\B}(\qDB(1))&=\lim_{n\to +\infty} \theta_{\B}(\DB(x_n))\\
&=\lim_{n\to +\infty} \LB(x_{\B}-x_n)\\
&=\qLB(x_{\B}-1).\qedhere
\end{align*}
\end{proof}

\begin{example}\label{Ex : AlternateBaseQuasiLazy}
Consider the alternate base from Example~\ref{Ex : AlternateBase}. We have $\qDB(1)=200(10)^\omega$, $\qDBi{1}(1)=(10)^\omega$, $\qLB(x_{\B}-1)=012(02)^\omega$ and $\LBi{1}(x_{\B^{(1)}}-1)=(02)^\omega$.
\end{example}

\begin{proposition}\label{Pro : QuasiLazyReprOfXB-1}
The quasi-lazy expansion $\qLB(x_{\B}-1)$ is a $\B$-representation of $x_{\B}-1$.
\end{proposition}
\begin{proof}
This is direct by Proposition~\ref{Pro : LinkQuasiGreedyQuasiLazy},~\cite[Proposition 22]{CharlierCisternino2021} and~\eqref{Eq : ValTheta}.
\end{proof}

Note that, in comparison with the quasi-greedy $\B$-expansion of $1$ which is always infinite, the quasi-lazy $\B$-expansion of $x_{\B}-1$ can be finite. 

\begin{example}\label{Ex : QuasiGreedyDiffQUasiLazyEqualIntegers}
Consider an alternate base $\B=\overline{(\beta_0,\ldots,\beta_{p-1})}$ such that for all $i\in\Int$, $\beta_i \in \N_{\ge 2}$. We get $\qDBi{i}(1)=((\beta_i-1)\cdots (\beta_{p-1}-1)(\beta_0-1)\cdots (\beta_{i-1}-1))^\omega$ and since $x_{\B^{(i)}}=1$ for all $i\in \Int$, we have $\qLBi{i}(0)=0^\omega$.
\end{example}

The following result gives a necessary condition on the Cantor base $\B$ to have a finite quasi-lazy $\B$-expansion of $x_{\B}-1$.

\begin{proposition}\label{Pro : QuasiLazyFinite}
If the quasi-lazy $\B$-expansion of $x_{\B}-1$ is finite of length $n\in \N$, then $x_{\B^{(n)}}=1$.
\end{proposition}
\begin{proof}
Suppose that $\qLB(x_{\B}-1) =\ell_0\cdots \ell_{n-1} 0^\omega$ with $n\in \N$ and $\ell_{n-1}\ne 0$ (if it exists, that is if $n\ne 0$). By Proposition~\ref{Pro : LinkQuasiGreedyQuasiLazy}, we get that 
\[
\qDB(1)=(\ceil{\beta_0}-1-\ell_0)\cdots (\ceil{\beta_{n-1}}-1-\ell_{n-1})(\ceil{\beta_n}-1) (\ceil{\beta_{n+1}}-1)\cdots .\]
However, by~\cite[Proposition 30]{CharlierCisternino2021}, we know that \[\sigma^n(\qDB(1))=(\ceil{\beta_n}-1) (\ceil{\beta_{n+1}}-1)\cdots \le_{\lex} \qDBi{n}(1).\]
Hence, we obtain that $\sigma^n(\qDB(1))=\qDBi{n}(1)$. We conclude that 
\begin{align*}
x_{\B ^{(n)}}&=\val_{\B^{(n)}}(\sigma^n(\qDB(1)))\\
&= \val_{\B^{(n)}}(\qDBi{n}(1)) \\
&=1.\qedhere
\end{align*}
\end{proof}

\begin{corollary}
If the quasi-lazy $\B$-expansion of $x_{\B}-1$ is finite of length $n\in \N$, then $\beta_k\in \N_{\ge2}$ for all $k\ge n$.
\end{corollary}

As the following example shows, the necessary conditions given by the previous proposition and corollary are not sufficient.
\begin{example}
Consider the Cantor base $\B=(\frac43,2,2,2,2,2\cdots)$. We have $x_{\B}=\frac32$ and $x_{\B^{(n)}}=1$ for all $n\ge 1$. However, we have $\qDB(1)=(10)^\omega$ and $\qLB(x_{\B}-1)=(01)^\omega$.
\end{example}

An infinite word in $\prod_{n\in \N}[\![0 , \ceil{\beta_n}-1]\!]$ is said \emph{ultimately maximal} if there exists $N\in \N$ such that for all $n\ge N$, the $(n+1)^\text{st}$ letter of $\qLB(x_{\B}-1)$ is $\ceil{\beta_n}-1$.

\begin{lemma}\label{Lem : UltimatelyMaximal}
The infinite word $\qLB(x_{\B}-1)$ cannot be ultimately maximal. 
\end{lemma}
\begin{proof}
This is a direct consequence of Proposition~\ref{Pro : LinkQuasiGreedyQuasiLazy} since $\qDB(1)$ is infinite.
\end{proof}

We now prove that  $\qLB(x_{\B}-1)$ is lexicographically smaller than all $\B$-representations of real numbers in $(x_{\B}-1,x_{\B}]$ belonging to $\prod_{n\in \N}[\![0 , \ceil{\beta_n}-1]\!]$.

\begin{proposition}\label{Pro : QuasiLazySmaller}
If $a$ is an infinite word in $\prod_{n\in \N}[\![0 , \ceil{\beta_n}-1]\!]$ such that $\val_{\B}(a)\in (x_{\B}-1,x_{\B}]$, then $a>_{\lex} \qLB(x_{\B}-1)$. 
\end{proposition}
\begin{proof}
Let $a$ be an infinite word in $\prod_{n\in \N}[\![0 , \ceil{\beta_n}-1]\!]$ such that $\val_{\B}(a)\in (x_{\B}-1,x_{\B}]$. Then $\theta_{\B}(a)$ is an infinite word over $\prod_{n\in \N}[\![0 , \ceil{\beta_n}-1]\!]$ and by~\eqref{Eq : ValTheta}, we have $\val_{\B}(\theta_{\B}(a))=x_{\B}-\val_{\B}(a)\in [0,1)$. By~\cite[Proposition 23]{CharlierCisternino2021}, we get that $\theta_{\B}(a)<_{\lex} \qDB(1)$. Moreover, by Proposition~\ref{Pro : LinkQuasiGreedyQuasiLazy}, we have $\qDB(1)=\theta_{\B}(\qLB(x_{\B}-1))$. Hence, by~\eqref{Eq : ThetaDecreasing}, we conclude that $a>_{\lex} \qLB(x_{\B}-1)$.
\end{proof}

Note that, similarly to Proposition~\ref{Pro : LazyLexMin}, Proposition~\ref{Pro : QuasiLazySmaller} is weaker than its analogous greedy one~\cite[Proposition 23]{CharlierCisternino2021} since we fix the alphabet of the $\B$-representations. A stronger result cannot be stated as illustrated in the next example. 

\begin{example}
Continuing Examples~\ref{Ex : AlphabetImportant} and~\ref{Ex : AlternateBaseQuasiLazy}, the infinite word $003330^\omega$ is a $\B$-representa\-tion of $8-2\sqrt{13}$. However $003330^\omega <_{\lex} 012(02)^\omega=\qLB(x_{\B}-1)$.
\end{example}

By~\cite[Proposition 23]{CharlierCisternino2021}, the word $\qDB(1)$ is lexicographically maximal among all infinite $\B$-representations of all real numbers in $[0,1]$. 
The following result gives the translation of this property in terms of the lazy representations. 
 
\begin{proposition}
The quasi-lazy $\B$-expansion of  $x_{\B}-1$ is the lexicographically least $\B$-representation of $x_{\B}-1$ in $\prod_{n\in \N}[\![0 , \ceil{\beta_n}-1]\!]$ that is not ultimately maximal.
\end{proposition}
\begin{proof}
By Proposition~\ref{Pro : QuasiLazyReprOfXB-1} and Lemma~\ref{Lem : UltimatelyMaximal}, the quasi-lazy $\B$-expansion of  $x_{\B}-1$ is a $\B$-representation of $x_{\B}-1$ in $\prod_{n\in \N}[\![0 , \ceil{\beta_n}-1]\!]$ which is not ultimately maximal. 
Moreover, let $a$ be an infinite word in $\prod_{n\in \N}[\![0 , \ceil{\beta_n}-1]\!]$ such that $\val_{\B}(a)=x_{\B}-1$ and suppose that $a<_{\lex} \qLB(x_{\B}-1)$. As above, we get $\theta_{\B}(a) >_{\lex} \qDB(1)$ with $\val_{\B}(\theta_{\B}(a))=1$. By~\cite[Proposition 23]{CharlierCisternino2021}, the word $\theta_{\B}(a)$ must be a finite $\B$-representation of $1$. By setting $N$ to the length of the longest prefix of $\theta_{\B}(a)$ not ending with $0$, we get $a_n=\ceil{\beta_n}-1$ for all $n\ge N$, that is $a$ is ultimately maximal in $\prod_{n\in \N}[\![0 , \ceil{\beta_n}-1]\!]$.
\end{proof}

\section{Admissible sequences}\label{Sec : Admissible}

We let $D'_{\B}$ denote the subset of $A_{\B}^\N$ of all lazy $\B$-expansions of real numbers in the interval $(x_{\B}-1,x_{\B}]$ and let $S'_{\B}$ denote the topological closure of $D'_{\B}$ with respect to the prefix distance of infinite words:
\[
	D'_{\B}=\{\LB(x)\colon x\in(x_{\B}-1,x_{\B}]\} \quad \mathrm{and}\quad S'_{\B}=\overline{D'_{\B}}.
\]  

The following result links the sets $D'_{\B}$ and $S'_{\B}$ with their analogous greedy ones $D_{\B}=\{\DB(x)\colon x\in [0,1)\}$ and $S_{\B}=\overline{D_{\B}}$ from~\cite{CharlierCisternino2021}. 

\begin{proposition}\label{Pro : ThetaDBSB}
The maps $\restr{\theta_{\B}}{D_{\B}}\colon D_{\B} \to D'_{\B}$ and $\restr{\theta_{\B}}{S_{\B}}\colon S_{\B} \to S'_{\B}$ are both bijective.
\end{proposition}
\begin{proof}
By Proposition~\ref{Pro : LinkGreedyLazy}, the map $\restr{\theta_{\B}}{D_{\B}}$ is well defined and surjective. 
Hence, by continuity of the map $\theta_{\B}$, the map $\restr{\theta_{\B}}{S_{\B}}$ is also well defined and surjective. 
Moreover, since the map $\theta_{\B}$ is injective, so are the maps $\restr{\theta_{\B}}{D_{\B}}$ and $\restr{\theta_{\B}}{S_{\B}}$. 
\end{proof}

Note that, in the particular case of alternate bases, Proposition~\ref{Pro : ThetaDBSB} can be deduced from~\cite[Remark 6.3]{CharlierCisterninoDajani2021}. 

\begin{proposition}
Let $a,b \in S'_{\B}$.
\begin{enumerate}
\item If $a<_{\lex} b$ then $\val_{\B}(a)\leq\val_{\B}(b)$.
\item If $\val_{\B}(a)<\val_{\B}(b)$ then $a<_{\lex} b$.
\end{enumerate}
\end{proposition}
\begin{proof}
Suppose that $a,b \in S'_{\B}$ are such that $a<_{\lex} b$. By Proposition~\ref{Pro : ThetaDBSB} and~\eqref{Eq : ThetaDecreasing}, we have $\theta_{\B}(a),\theta_{\B}(b)\in S_{\B}$ and $\theta_{\B}(a)>_{\lex} \theta_{\B}(b)$. By~\cite[Proposition 31]{CharlierCisternino2021}, we $\val_{\B}(\theta_{\B}(a))\ge \val_{\B}(\theta_{\B}(b))$. We conclude the proof of the first item by~\eqref{Eq : ValTheta}. The second item immediately follows.
\end{proof}

We are now able to state a Parry-like theorem for Cantor real bases in the lazy framework. 

\begin{theorem}\label{Thm : LazyParry}
Let $a$ be an infinite word over $\N$.
\begin{enumerate}
\item The word $a$ belongs to $D'_{\B}$ if and only if $a\in \prod_{n\in \N}[\![0 , \ceil{\beta_n}-1]\!]$ and for all $n\in \N$, 
\begin{align*}
\sigma^n(a)>_{\lex} \qLBi{n}(x_{\B^{(n)}}-1).
\end{align*}
\item The word $a$ belongs to $S'_{\B}$ if and only if $a\in \prod_{n\in \N}[\![0 , \ceil{\beta_n}-1]\!]$ and for all $n\in \N$, \begin{align*}
\sigma^n(a)\ge_{\lex} \qLBi{n}(x_{\B^{(n)}}-1).
\end{align*}
\end{enumerate}
\end{theorem}
\begin{proof}
Let $a$ be an infinite word.
We have $a\in D'_{\B}$ if and only if  $a\in \prod_{n\in \N}[\![0 , \ceil{\beta_n}-1]\!]$ and $\theta_{\B}(a)\in D_{\B}$. 
Moreover, by~\cite[Theorem 26]{CharlierCisternino2021}, we have $\theta_{\B}(a)\in D_{\B}$ if and only if $\sigma^n(\theta_{\B}(a))<_{\lex} \qDBi{n}(1)$ for all $n\in \N $. 
However, for all $n\in \N$, by Lemma~\ref{Lem : CommuteShiftTheta}, we have $ \sigma^n(\theta_{\B}(a))=\theta_{\B^{(n)}}(\sigma^n(a))$ and by Proposition~\ref{Pro : LinkQuasiGreedyQuasiLazy}, we have $\qDBi{n}(1)= \theta_{\B^{(n)}}(\qLBi{n}(x_{\B^{(n)}}-1))$. 
Hence, the first item follows from~\eqref{Eq : ThetaDecreasing}.
The second item can be proved in a similar fashion by using~\cite[Proposition 30]{CharlierCisternino2021}.
\end{proof}

\begin{example}
Consider $\B=(\overline{\frac{1+ \sqrt{13}}{2},\frac{5+ \sqrt{13}}{6}})$ from Example~\ref{Ex : AlternateBase}. In view of Example~\ref{Ex : AlternateBaseQuasiLazy}, the sequence $a=(2120)^\omega$ belongs to $D'_{\B}$.
\end{example}

Note that in Theorem~\ref{Thm : LazyParry}, the hypothesis that $a$ belongs to $\prod_{n\in \N}[\![0 , \ceil{\beta_n}-1]\!]$ is required. For otherwise, any sequence $a$ such that $a_n>\ceil{\beta_n}-1$ for all $n\in \N$ would belong to $D_{\B}'$.

As a consequence of Theorem~\ref{Thm : LazyParry}, we can characterize the set $D'_{\B}$ by translating~\cite[Proposition 34 and Corollaries 35 and 36]{CharlierCisternino2021} to the lazy framework. 
To do so, we define sets of finite words $X'_{\B,n}$ for $n\in\N_{\ge 1}$ as follows. If $\qLB(x_{\B}-1)=\ell_0\ell_1\cdots$ then, for all $n\in \N_{\ge1}$, we let
\[
	X'_{\B,n}=\{\ell_0\cdots \ell_{n-2}s \colon s\in[\![\ell_{n-1}+1, \ceil{\beta_{n-1}}-1]\!]\}.
\]
Note that $X'_{\B,n}$ is empty if and only if $\ell_{n-1}=\ceil{\beta_{n-1}}-1$.

\begin{proposition}
\label{Pro : DX} 
We have
\[
	D'_{\B}=\bigcup_{n_0\in\N_{\ge 1}} X'_{\B,n_0}
		\Bigg(\bigcup_{n_1\in\N_{\ge 1}} X'_{\B^{(n_0)},n_1}
		\Bigg(\bigcup_{n_2\in\N_{\ge 1}} X'_{\B^{(n_0+n_1)},n_2}
		\Bigg(
		\quad\cdots\quad
		\Bigg)
		\Bigg)
		\Bigg).
\] 
Therefore, we have $D'_{\B}=\displaystyle{\bigcup_{n\in\N_{\ge 1}} X'_{\B,n} D'_{\B^{(n)}}}$ and any prefix of $\qLB(x_{\B}-1)$ belongs to $\Pref(D'_{\B})$.
\end{proposition}
\begin{proof}
This follows from Propositions~\ref{Pro : LinkQuasiGreedyQuasiLazy},~\ref{Pro : ThetaDBSB} and~\cite[Proposition 34]{CharlierCisternino2021} since $w_0w_1\cdots w_{n-1}\in X'_{\B,n}$ if and only if  $(\ceil{\beta_0}-1-w_0)(\ceil{\beta_1}-1-w_1)\cdots (\ceil{\beta_{n-1}}-1-w_{n-1})\in X_{\B,n}$.
\end{proof}

As in the greedy case, for lazy expansions in alternate bases, Proposition~\ref{Pro : DX} can be straightened as follows. 
Consider an alternate base $\B$ of length $p$.
We define sets of finite words $Y'_{\B,h}$ for $h\in\Int$ as follows. If $\qLB(x_{\B}-1)=\ell_0\ell_1\cdots$ then, for all $h\in\Int$, we let
\[
	Y'_{\B,h}=\{\ell_0\cdots \ell_{n-2}s \colon 
	n\in\N_{\ge 1},\
	n\bmod p= h,\
	s\in[\![\ell_{n-1}+1, \ceil{\beta_{n-1}}-1]\!]
	\}.
\]
Note that $Y'_{\B,h}$ is empty if and only if for all $n\in\N_{\ge 1}$ such that $n\bmod p= h$, $\ell_{n-1}=\ceil{\beta_{n-1}}-1$. Moreover, unlike the sets $X'_{\B,n}$ defined above, the sets $Y'_{\B,h}$ can be infinite. 

\begin{proposition}
\label{Pro : DY} 
Let $\B$ be an alternate base of length $p$. We have 
\[
	D'_{\B}=\bigcup_{h_0=0}^{p-1} Y'_{\B,h_0}
		\Bigg(\bigcup_{h_1=0}^{p-1} Y'_{\B^{(h_0)},h_1}
		\Bigg(\bigcup_{h_2=0}^{p-1} Y'_{\B^{(h_0+h_1)},h_2}
		\Bigg(
		\quad\cdots\quad
		\Bigg)
		\Bigg)
		\Bigg).
\]
Therefore, we have $D'_{\B}=\displaystyle{\bigcup_{h=0}^{p-1} Y'_{\B,h} D'_{\B^{(h)}}}$.
\end{proposition}

\section{The lazy $\B$-shift}\label{Sec : LazyShift}

This section is concerned with the study of the lazy $\B$-shift.
First, let us define
\[
	\Delta'_{\B}=\bigcup_{n\in\N}D'_{\B^{(n)}}
	\quad\text{and}\quad
	 \Sigma'_{\B}=\overline{\Delta'_{\B}}.
\] 
By Proposition~\ref{Pro : ThetaDBSB}, we get 
\begin{equation}
\Delta'_{\B}=\bigcup_{n\in \N}\theta_{\B^{(n)}}(D_{\B^{(n)}}).\label{Eq : DeltaUnionTheta}
\end{equation}

\begin{proposition}
\label{Pro : Subshift}
The sets $\Delta'_{\B}$ and $\Sigma'_{\B}$ are both shift-invariant. 
\end{proposition}
\begin{proof}
Let $a$ be an infinite word over $\N$. By~\eqref{Eq : DeltaUnionTheta}, if $a$ belongs to $\Delta'_{\B}$, then there exists $n\in \N$ and an infinite word $b\in D_{\B^{(n)}}$ such that $a=\theta_{\B^{(n)}}(b)$. We obtain that $\sigma(a)=\sigma(\theta_{\B^{(n)}}(b))=\theta_{\B^{(n+1)}}(\sigma(b))$ by Lemma~\ref{Lem : CommuteShiftTheta}. By~\cite[Theorem 26]{CharlierCisternino2021}, $\sigma(b)\in D_{\B^{(n+1)}}$ so $\sigma(a)\in D'_{\B^{(n+1)}}$. Then, it is easily seen that if $a\in S'_{\B^{(n)}}$ then $\sigma(a)\in S'_{\B^{(n+1)}}$.
\end{proof}

Since the set $\Sigma'_{\B}$ is shift-invariant and closed with respect to the topology induced by the prefix distance on infinite words, we conclude that the subset $\Sigma'_{\B}$ of $A_{\B}^{\N}$ is a subshift, which we call the \emph{lazy ${\B}$-shift}. 

\begin{remark}
It is important to remark that the lazy $\B$-shift is not the lazy $\B$-shift defined in~\cite{CharlierCisterninoDajani2021}. In fact, as said in~\cite[Remark 36]{CharlierCisterninoDajani2021}, there is two ways to extend the notion of $\beta$-shift from the real base case to the alternate bases or more generally to the Cantor base framework.
\end{remark}

Recall that the set of finite factors and the set of prefixes of all elements in a language $L$ are respectively denoted $\Fac(L)$ and $\Pref(L)$. 
Let us now study the factors of the lazy $\B$-shift.

\begin{proposition}
\label{prop:Fac}
We have $\Fac(D'_{\B})=\Fac(\Delta'_{\B})=\Fac(\Sigma'_{\B})$.
\end{proposition}
\begin{proof}
By definition, we have $\Fac(D'_{\B})\subseteq\Fac(\Delta'_{\B})=\Fac(\Sigma'_{\B})$. It remains to show that $\Fac(D'_{\B})\supseteq\Fac(\Delta'_{\B})$. Let $f\in \Fac(\Delta'_{\B})$. By~\eqref{Eq : DeltaUnionTheta}, there exist $n\in \N$ and $b\in D_{\B^{(n)}}$ such that $f\in \Fac(\theta_{\B^{(n)}}(b))$. In particular, $f\in \Fac(\theta_{\B}(0^nb))$ where, by~\cite[Theorem 26]{CharlierCisternino2021}, $0^nb\in D_{\B}$. We obtain that $f\in \Fac(\theta_{\B}(D_{\B}))=\Fac(D'_{\B})$ by Proposition~\ref{Pro : ThetaDBSB}.
\end{proof}

\begin{corollary}\label{Cor : SigmaUnionThetaPref}
We have 
\[
\Fac(\Sigma'_{\B})=\bigcup_{n\in \N} \theta_{\B^{(n)}}\big( \Pref(D_{\B^{(n)}}) \big).
\]
\end{corollary}
\begin{proof}
By Propositions~\ref{Pro : Subshift} and~\ref{prop:Fac}, we have
$\Fac(\Sigma'_{\B})=\Pref(\Delta'_{\B})=\bigcup_{n\in \N} \Pref(D'_{\B^{(n)}}).$ The conclusion follows from Proposition~\ref{Pro : ThetaDBSB}.
\end{proof}

In the alternate base framework, an analogue of Bertrand-Mathis' theorem~\cite{Bertrand-Mathis1989} can be stated for the lazy $\B$-shift. To do so, recall that a subshift $S$ of $A^{\N}$ is called \emph{sofic} if the language $\Fac(S)\subseteq A^*$ is accepted by a finite automaton. 

\begin{theorem}\label{Thm : SoficLazy}
Let $\B$ be an alternate base of length $p$. The lazy $\B$-shift $\Sigma'_{\B}$ is sofic if and only if for all $i\in \Int$, $\qLBi{i}(x_{\B^{(i)}}-1)$ is ultimately periodic.
\end{theorem}

In order to prove this result, let us construct an automaton $\mathcal{A}'_{\B}$ in the case where all quasi-lazy expansions are ultimately periodic and state some results in order to link this automaton with the one used in the greedy case (see~\cite[Theorem 48]{CharlierCisternino2021}) called $\mathcal{A}_{\B}$. Roughly, if all the quasi-lazy expansions are ultimately periodic, then so are the quasi-greedy expansions and the ``image'' of the automaton $\mathcal{A}_{\B}$ under the maps $\theta_{\B^{(i)}}$ with $i\in \Int$ is an automaton accepting $\Fac(\Sigma'_{\B})$. This notion of ``image'' of the automaton under the maps $\theta_{\B^{(i)}}$ will be clearer in what follows, more precisely in Lemmas~\ref{Lem : TransitionsGreedyLazyIIF} and~\ref{Lem : AutomataTheta}.

Henceforth, let $\B$ be an alternate base of length $p$ and suppose that for all $i\in \Int$, $\qLBi{i}(x_{\B^{(i)}}-1)$ is ultimately periodic and write\footnote{Recall that $\qLBi{i}(x_{\B^{(i)}}-1)$ can be finite, hence, $n_i$ can be equal to $1$ and $\ell_{m_i}^{(i)}=0$.} 
 \[
\qLBi{i}(x_{\B^{(i)}}-1)=\ell_0^{(i)}\cdots \ell_{m_i-1}^{(i)} \big(\ell_{m_i}^{(i)}\cdots \ell_{m_i+n_i-1}^{(i)}\big)^\omega.
\] 
Without loss of generality, from now on, suppose that $n_i$ is a multiple of $p$ (it suffices to take the least common multiple of $p$ and the length of the period).
For all $i\in \Int$, by Proposition~\ref{Pro : LinkQuasiGreedyQuasiLazy}, we get\footnote{Note that the preperiod and period $m_i$ and $n_i$ may be not minimal.}
 \[
\qDBi{i}(1)=t_0^{(i)}\cdots t_{m_i-1}^{(i)} \big(t_{m_i}^{(i)}\cdots t_{m_i+n_i-1}^{(i)}\big)^\omega
\] 
with $t^{(i)}_n=\ceil{\beta_{i+n}}-1-\ell^{(i)}_n$ for all $n\in [\![0,m_i+n_i-1]\!]$.
Hence, all quasi-greedy expansions of $1$ are ultimately periodic. Let $\mathcal{A}_{\B}$ be the automaton over the alphabet $A_{\B}$
from~\cite[Section 7.3]{CharlierCisternino2021} which accepts $\Fac(\Sigma_{\B})$ (see~\cite[Theorem 48]{CharlierCisternino2021}).
Recall that $\mathcal{A}_{\B}=(Q,I,F,A_{\B},\delta)$ where
\begin{align*}
	Q&=\big\{q_{i,j,k}\colon 
		i,j\in\Int,\ k\in[\![0,m_i+n_i-1]\!]\big\},\\
	I\, &=\big\{q_{i,i,0}\colon i\in\Int\big\},\\
	F&=Q
\end{align*}
and, for each $i,j\in\Int$ and each $k\in[\![0,m_i+n_i-1]\!]$, we have
\[
	\delta(q_{i,j,k},t_k^{(i)})=
	\begin{cases}
		q_{i,(j+1)\bmod p,k+1} & \text{ if }k\ne m_i+n_i-1\\
		q_{i,(j+1)\bmod p,m_i} & \text{ else}
	\end{cases}
\]
and for all $s\in[\![0,t_k^{(i)}-1]\!]$, we have
\[
	\delta(q_{i,j,k},s)=q_{(j+1)\bmod p,(j+1)\bmod p,0}.
\]
Define the automaton $\mathcal{A}'_{\B}=(Q,I,F,A_{\B},\delta')$ where for each $i,j\in\Int$ and each $k\in[\![0,m_i+n_i-1]\!]$, 
we have 
\begin{equation}
		\label{Eq : TransitionFirstPossibility}
\delta'(q_{i,j,k},\ell_k^{(i)})=
\begin{cases}
q_{i,(j+1)\bmod p,k+1} & \text{if } k\ne m_i+n_i-1 
		\\
		q_{i,(j+1)\bmod p,m_i} & \text{else}
\end{cases}
\end{equation}
and for all $s\in[\![\ell_k^{(i)}+1,\ceil{\beta_{j}}-1]\!]$, we have
\begin{equation}
\label{Eq : TransitionThirdPossibility}
	\delta'(q_{i,j,k},s)=q_{(j+1)\bmod p,(j+1)\bmod p,0}.
\end{equation}

Since we supposed that the parameters $n_i$, with $i\in \Int$, were multiples of $p$, we get the following result.

\begin{lemma}
In the automata $\mathcal{A}_{\B}$ and $\mathcal{A}'_{\B}$, for all $i,j\in \Int$ and $k\in [\![0,m_i+n_i-1]\!]$, the state $q_{i,j,k}$ is accessible if and only if $i+k\equiv j \pmod p$.
\end{lemma}
\begin{proof}
Let us prove the result for the automaton $\mathcal{A}'_{\B}$. The reasoning for the automaton $\mathcal{A}_{\B}$ is similar.
Suppose that $i+k\equiv j \pmod p$. There exists a path from $q_{i,i,0}$ to $q_{i,j,k}$ labeled by $\ell_0^{(i)}\cdots \ell_{k}^{(i)}$. In fact, for all $k'\in [\![0,k-1]\!]$, we have 
\begin{equation}\label{Eq : AccessibleJ}
\delta'(q_{i,(i+k')\bmod p,k'}, \ell_{k'}^{(i)})=q_{i,(i+k'+1)\bmod p,k'+1}.
\end{equation}
Conversely, let $i,j\in \Int$ and $k\in [\![0,m_i+n_i-1]\!]$.
Suppose that the state $q_{i,j,k}$ is accessible. 
Let $c$ be an initial path ending in $q_{i,j,k}$. By definition of the transitions, if a path starts in $q_{i',i',0}$ with $i' \in \Int \setminus \{i\}$ and ends in $q_{i,j,k}$ then it necessarily goes through $q_{i,i,0}$ by using a transition of the form~\eqref{Eq : TransitionThirdPossibility}. Hence, we may suppose that the path $c$ only uses transitions of the form~\eqref{Eq : TransitionFirstPossibility}. The conclusion follows since for all $k'\in [\![0,k-1]\!]$, we have~\eqref{Eq : AccessibleJ} and 
\begin{align*}
\delta'(q_{i,(i+m_i+n_i-1)\bmod p,m_i+n_i-1},\ell_{m_i+n_i-1}^{(i)})&=q_{i,(i+m_i+n_i)\bmod p,m_i}
\end{align*}
where $n_i\equiv 0 \pmod p$ by assumption.
\end{proof}

By the previous lemma, from now on, we consider the automata $\mathcal{A}_{\B}$ and $\mathcal{A}'_{\B}$ by preserving only the set
\[
\big\{q_{i,(i+k)\bmod p,k}\colon i\in\Int,\ k\in[\![0,m_i+n_i-1]\!]\big\}
\]
of accessible states and we keep the same notation.

\begin{lemma}\label{Lem : TransitionsGreedyLazyIIF}
Let $a\in A_{\B}$, $i_1,i_2\in \Int$ and $k_1\in[\![0,m_{i_1}+n_{i_1}-1]\!],k_2\in [\![0,m_{i_2}+n_{i_2}-1]\!]$. We have
\[
\delta(q_{i_1,(i_1+k_1)\bmod p,\, k_1},a)=q_{i_2,(i_2+k_2)\bmod p,\, k_2}\]
 if and only if 
 \[
 \delta'(q_{i_1,(i_1+k_1)\bmod p,\, k_1},\ceil{\beta_{i_1+k_1}}-1-a)=q_{i_2,(i_2+k_2)\bmod p,\, k_2}. \]
\end{lemma}
\begin{proof}
Fix $a\in A_{\B}$, $i\in \Int$ and $k\in[\![0,m_{i}+n_{i}-1]\!]$.
By definition of the automaton $\mathcal{A}_{\B}$, from $q_{i,(i+k)\bmod p,\, k}$ we have the following transitions
\[
\delta(q_{i,(i+k)\bmod p,\, k},a)=
\begin{cases}
q_{i,(i+k+1)\bmod p,\, k+1} & \text{if } a=t_{k}^{(i)} \text{ and } k\ne m_i+n_i-1\\
q_{i,(i+m_i)\bmod p,\, m_i} & \text{if } a=t_{k}^{(i)} \text{ and } k= m_i+n_i-1\\
q_{(i+k+1)\bmod p,(i+k+1)\bmod p,\,  0} & \text{if } a\in [\![0, t_k^{(i)}-1]\!].
\end{cases}
\]
Similarly, by definition of $\mathcal{A}'_{\B}$,
we have
\[
\delta'(q_{i,(i+k)\bmod p,\, k},a)=
\begin{cases}
q_{i,(i+k+1)\bmod p,\, k+1} & \text{if } a=\ell_{k}^{(i)} \text{ and } k\ne m_i+n_i-1\\
q_{i,(i+m_i)\bmod p,\, m_i} & \text{if } a=\ell_{k}^{(i)} \text{ and } k= m_i+n_i-1\\
q_{(i+k+1)\bmod p,(i+k+1)\bmod p,\,  0} & \text{if } a\in [\![ \ell_k^{(i)}+1,\ceil{\beta_{i+k}}-1]\!].
\end{cases}
\]
We get the conclusion since $\ell_{k}^{(i)}=\ceil{\beta_{i+k}}-1-t_{k}^{(i)}$, and hence $a\in [\![0, t_k^{(i)}-1]\!]$ if and only if $
\ceil{\beta_{i+k}}-1-a\in[\![\ell_k^{(i)}+1,\ceil{\beta_{i+k}}-1]\!].$
\end{proof}

\begin{example} \label{Ex : AlternateSoficAutomaton}
Let $\B=(\overline{\varphi^2, 3+\sqrt{5}})$. We have $\DB(1)=2(30)^\omega$, $\DBi{1}(1)=5(03)^\omega$ and  $\LB(x_{\B}-1)=02^\omega$, $\LBi{1}(x_{\B^{(1)}}-1)=02^\omega$. The corresponding accessible automata $\mathcal{A}_{\B}$ and $\mathcal{A}'_{\B}$ are depicted in Figure~\ref{fig:Automaton-230-503-accessible} with red and blue labels respectively.
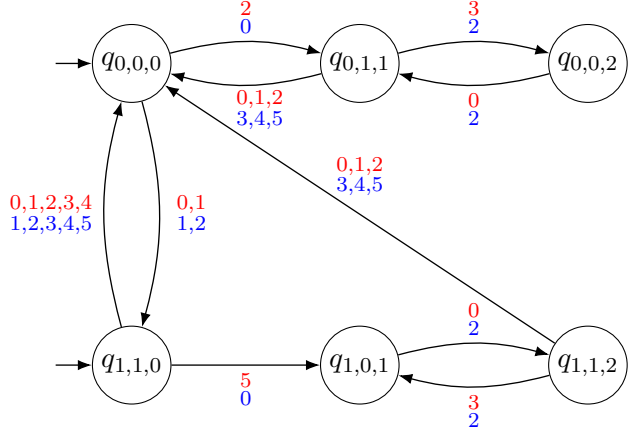
\begin{figure}
\centering
\begin{tikzpicture}
\tikzstyle{every node}=[shape=circle, fill=none, draw=black,
minimum size=20pt, inner sep=2pt]
\node(1) at (0,4) {$q_{0,0,0}$};
\node(3) at (6,4) {$q_{0,0,2}$};
\node(5) at (3,4) {$q_{0,1,1}$};

\node(8) at (3,0) {$q_{1,0,1}$};
\node(10) at (0,0) {$q_{1,1,0}$};
\node(12) at (6,0) {$q_{1,1,2}$};

\tikzstyle{every node}=[shape=circle, fill=none, draw=black, minimum size=15pt, inner sep=2pt]
\tikzstyle{every path}=[color=black, line width=0.5 pt]
\tikzstyle{every node}=[shape=circle, minimum size=5pt, inner sep=2pt]
\draw [-Latex] (-1,4) to node [above] {$ $} (1); 
\draw [-Latex] (-1,0) to node [above] {$ $} (10); 

\draw [-Latex] (1) to [bend left=15] node [above] {$\substack{{\color{red}2}\\ {\color{blue}0}}$} (5); 
\draw [-Latex] (1) to [bend left=15] node [right] {$\substack{{\color{red}0,1}\\{\color{blue}1,2}} $} (10);

\draw [-Latex] (3) to [bend left=15] node [below] {$\substack{{\color{red}0}\\{\color{blue}2}}$} (5);

\draw [-Latex] (5) to [pos=0.4,bend left=15] node [below=-0.17] {$\substack{{\color{red}0,1,2}\\{\color{blue}3,4,5}}$} (1); 
\draw [-Latex] (5) to [bend left=15] node [above] {$\substack{{\color{red}3}\\{\color{blue}2}}$} (3);

\draw [-Latex] (8) to [bend left=15] node [above] {$\substack{{\color{red}0}\\{\color{blue}2}}$} (12); 

\draw [-Latex] (10) to  node [below] {$\substack{{\color{red}5}\\{\color{blue}0}}$} (8); 
\draw [-Latex] (10) to [bend left=15] node [left] {$\substack{{\color{red}0,1,2,3,4}\\{\color{blue}1,2,3,4,5}}$} (1);

\draw [-Latex] (12) to [bend left=15] node [below] {$\substack{{\color{red}3}\\{\color{blue}2}}$} (8); 
\draw [-Latex] (12) to node [above] {$\substack{{\color{red}0,1,2}\\{\color{blue}3,4,5}}$} (1); 
\end{tikzpicture}
\caption{An accessible automaton accepting $\Fac(\Sigma_{(\overline{\varphi^2, 3+\sqrt{5}})})$ (labels above and red) and  $\Fac(\Sigma'_{(\overline{\varphi^2, 3+\sqrt{5}})})$ (labels below and blue).}
\label{fig:Automaton-230-503-accessible}
\end{figure}
\end{example}

\begin{lemma}\label{Lem : AutomataTheta}
Let $i\in\Int$ and consider $w\in A_{\B}^\N$. The word $w$ is accepted in $\mathcal{A}_{\B}$ from $q_{i,i,0}$ if and only if $\theta_{\B^{(i)}}(w)$ is accepted in $\mathcal{A}'_{\B}$ from $q_{i,i,0}$.
\end{lemma}
\begin{proof}
This immediately follows from Lemma~\ref{Lem : TransitionsGreedyLazyIIF}.
\end{proof}

We are now ready to prove Theorem~\ref{Thm : SoficLazy}.

\begin{proof}[Proof of Theorem~\ref{Thm : SoficLazy}]
Suppose that, for all $i\in \Int$, $\qLBi{i}(x_{\B^{(i)}}-1)$ is ultimately periodic. 
For all $i\in \Int$, let
\[
\qLBi{i}(x_{\B^{(i)}}-1)=\ell_0^{(i)}\cdots \ell_{m_i-1}^{(i)} \big(\ell_{m_i}^{(i)}\cdots \ell_{m_i+n_i-1}^{(i)}\big)^\omega
\] 
with $n_i$ multiple of $p$. 
By Proposition~\ref{Pro : LinkQuasiGreedyQuasiLazy}, for all $i\in \Int$, we obtain 
 \[
\qDBi{i}(1)=t_0^{(i)}\cdots t_{m_i-1}^{(i)} \big(t_{m_i}^{(i)}\cdots t_{m_i+n_i-1}^{(i)}\big)^\omega
\] 
with $t^{(i)}_n=\ceil{\beta_{i+n}}-1-\ell^{(i)}_n$ for all $n\in [\![0,m_i+n_i-1]\!]$.
Let $\mathcal{A}_{\B}$ and $\mathcal{A}'_{\B}$ be the automata associated with the greedy and lazy expansions respectively. By~\cite[Theorem 48]{CharlierCisternino2021}, for each $i\in \Int$, the language accepted in $\mathcal{A}_{\B}$ from the initial state $q_{i,i,0}$ is precisely $\Pref(D_{\B^{(i)}})$. Hence, by Lemma~\ref{Lem : AutomataTheta},  in $\mathcal{A}'_{\B}$ the language accepted from the initial state $q_{i,i,0}$ is precisely $\theta_{\B^{(i)}}(\Pref(D_{\B^{(i)}}))$. We get the conclusion by Corollary~\ref{Cor : SigmaUnionThetaPref}.

Conversely, suppose that there exists $j\in\Int$ such that $\qLBi{j}(x_{\B^{(j)}}-1)$ is not ultimately periodic, then we prove that $\Sigma'_{\B}$ is not sofic. This follows the same lines as in the greedy case (see~\cite[Theorem 48]{CharlierCisternino2021}). Hence, in the subsequent, the main ideas of the proof are given.
Let 
\[
	\qLBi{i}(x_{\B^{(i)}}-1)=\ell_0^{(i)}\ell_1^{(i)}\cdots\quad \text{ for every } i\in\Int.
\] 
We define a partition $(G_1,\ldots,G_q)$ of $\Int$ as follows. 
Let $r=\Card\{\qLBi{i}(x_{\B^{(i)}}-1)\colon i\in\Int\}$ and let $i_1,\ldots,i_r\in\Int$ be such that $\qLBi{i_1}(x_{\B^{(i_1)}}-1),\ldots,\qLBi{i_r}(x_{\B^{(i_r)}}-1)$ are pairwise distinct and $\qLBi{i_1}(x_{\B^{(i_1)}}-1)<_{\lex}\cdots<_{\lex}\qLBi{i_r}(x_{\B^{(i_r)}}-1)$. Let $q\in[\![1,r]\!]$ be the unique index such that $\qLBi{i_q}(x_{\B^{(i_q)}}-1)=\qLBi{j}(x_{\B^{(j)}}-1)$ where $\qLBi{j}(x_{\B^{(j)}}-1)$ is not ultimately periodic by assumption.
We set 
\[
	G_s=\{i\in\Int\colon \qLBi{i}(x_{\B^{(i)}}-1)=\qLBi{i_s}(x_{\B^{(i_s)}}-1)\}\quad \text{for }s\in[\![1,q-1]\!] 
\]
and 
\[
	G_q=\{i\in\Int\colon \qLBi{i}(x_{\B^{(i)}}-1)\ge_{\lex} \qLBi{j}(x_{\B^{(j)}}-1)\}.
\] 
For each $s\in[\![1,q-1]\!]$, we write $G_s=\{i_{s,1},\ldots,i_{s,\alpha_s}\}$ where $i_{s,1}<\ldots<i_{s,\alpha_s}$ and we use the convention that $i_{s,\alpha_s+1}=i_{s+1,1}$ for $s\le q-2$ and $i_{q-1,\alpha_{q-1}+1}=j$. Moreover, we let $g\in\N_{\ge 1}$ be such that for all $i,i'\in \Int$ such that $\qLBi{i}(x_{\B^{(i)}}-1)\ne \qLBi{i'}(x_{\B^{(i')}}-1)$, the length-$g$ prefixes of $\qLBi{i}(x_{\B^{(i)}}-1)$ and $\qLBi{i'}(x_{\B^{(i')}}-1)$ are distinct. Then, for $s\in[\![1,q-1]\!] $, we define $C_s$ to be the least $c\in\N_{\ge 1}$ such that $\ell^{(i_s)}_{g-1+c}<\ceil{\beta_{i_s+g-1+c}}-1$. Finally, let $N\in\N_{\ge 1}$ be such that $pN\ge \max\{g,C_1,\ldots,C_{q-1}\}$.

For all $m\in\N$, consider 
\[
	w^{(m)}
	=\left(
	\prod_{s=1}^{q-1} 
	\prod_{k=1}^{\alpha_s}
	\ell_0^{(i_s)}\cdots \ell_{g-1}^{(i_s)} 
	(\lceil \beta_{i_{s,k}+g}\rceil-1)\cdots (\lceil \beta_{i_{s,k+1}+p(2N+1)-1}\rceil-1)	
	\right) 
	\ell_0^{(j)}\cdots \ell_{m-1}^{(j)}.
\] 
Now, let $m,n\in\N$ be distinct. Since $\qLBi{j}(x_{\B^{(j)}}-1)$ is not ultimately periodic, $\sigma^m \big(\qLBi{j}(x_{\B^{(j)}}-1)\big)\ne \sigma^n\big(\qLBi{j}(x_{\B^{(j)}}-1)\big)$. Thus, there exists $k\in\N_{\ge 1}$ such that $\ell_m^{(j)}\cdots \ell_{m+k-2}^{(j)}=\ell_n^{(j)}\cdots \ell_{n+k-2}^{(j)}$ and $\ell_{m+k-1}^{(j)}\ne \ell_{n+k-1}^{(j)}$. Without loss of generality, we suppose that $\ell_{m+k-1}^{(j)}<\ell_{n+k-1}^{(j)}$. Let $z=\ell_m^{(j)}\cdots \ell_{m+k-1}^{(j)}$. 
Similarly to proof of~\cite[Theorem 48]{CharlierCisternino2021}, it can be shown that $w^{(m)}z \in \Fac(\Sigma'_{\B})\cap \Pref(D'_{\B^{(i_{1,1})}})$ and $w^{(n)}z\notin \Fac(\Sigma'_{\B})$.
\end{proof}

\begin{remark}
In the proof of the necessary condition of Theorem~\ref{Thm : SoficLazy}, the parameters $\{r,i_1,\ldots,i_r,q,G_1,\ldots,G_q,\ldots\}$ may not coincide with the ones in the necessary condition of~\cite[Theorem 48]{CharlierCisternino2021}.
In fact, it may happen for example that there exist $i,j\in \Int$, such that
$\qDBi{i}(1)>_{\lex} \qDBi{j}(1)$ whereas $\qLBi{i}(x_{\B^{(i)}}-1)\le_{\lex}\qLBi{j}(x_{\B^{(j)}}-1)$. For instance, this is illustrated by Examples~\ref{Ex : QuasiGreedyDiffQUasiLazyEqualIntegers} and~\ref{Ex : AlternateSoficAutomaton}.
\end{remark}

\section{Acknowledgment}
The author thanks Jean-Pierre Schneiders for suggesting the way to approximate the value of $x_{\B}$ in Example~\ref{Ex : XBTM}.

The author is supported by the FNRS Research Fellow grant 1.A.564.19F.

\bibliographystyle{abbrv}
\bibliography{TheseBibliography}

\end{document}